\documentclass[12pt]{article}
\usepackage{fullpage}
\usepackage{epsf,epsfig,amsfonts,amsgen,amsmath,amstext,amsbsy,amsopn,amsthm,amssymb,epstopdf,tikz,mathrsfs}
\usepackage{comment}
\usepackage{ebezier,eepic,mathtools,dsfont,lmodern}
\usepackage{enumitem}
\usepackage{textcomp}
\usepackage{color,colordvi}
\usepackage{multirow}
\usepackage{url}
\usepackage{graphicx}
\usepackage{epsf}
\usepackage[format=hang, margin=10pt]{caption}
\usepackage{lmodern}
\usepackage{cite}
\newtheorem{theorem}{Theorem}[section]
\newtheorem{proposition}[theorem]{Proposition}
\newtheorem{corollary}[theorem]{Corollary}
\newtheorem{conjecture}[theorem]{Conjecture}
\newtheorem{lemma}[theorem]{Lemma}

\theoremstyle{definition}
\newtheorem{definition}[theorem]{Definition}
\newtheorem{example}[theorem]{Example}
\newtheorem{remark}[theorem]{Remark}

\newcommand{\cS}{\mathcal S}
\newcommand{\intervaltwo}[2]{[#1,#2]_2}

\begin{document}

\title{Generalized Duke's theorem for signed Graphs}

\author{Yichao Chen\thanks{School of Mathematics, SuZhou University of Science and Technology, Suzhou 215009, P.R. China. Email: chengraph@163.com.}\,
\qquad
Yan Yang\thanks{School of Mathematics and KL-AAGDM, Tianjin University, Tianjin 300354, P.R. China. Email: yanyang@tju.edu.cn.} \,}

\date{}

\maketitle

 \begin{abstract}
Duke's interpolation theorem states that the orientable genera of a connected graph form an integer interval, and Stahl established the corresponding result for nonorientable embeddings. In 1991, \v{S}ir\'a\v{n} showed that this interpolation property fails for signed graph embeddings: the Euler-genus spectrum of a signed graph may contain gaps. He subsequently asked whether all such gaps must occur at the lower end of the spectrum.

In this paper, we establish a   characterization of the Euler-genus spectrum of a connected signed graph. We prove that, for each parity class, the Euler genera form a step-two interval. Moreover, whenever both parity classes are nonempty, their maximum elements differ by one. As an  consequence, if two consecutive integers $k$ and $k+1$ belong to the Euler-genus spectrum, then every integer from $k$ to the maximum Euler-genus also belongs to the spectrum, thereby answering \v{S}ir\'a\v{n}'s question affirmatively. Our proof uses the pre-signed graph representation of signed embeddings together with ordered adjacent-exchange operations and a matching interpretation of face numbers.


\end{abstract}

\section{Introduction}
The central problem in topological graph theory is to embed a given graph $G$ on surfaces. A {\it surface} is a compact 2-dimensional manifold without boundary. All surfaces can be obtained from the sphere $\mathbb{S}_0$ by adding handles and crosscaps. If we add $g$ handles to $\mathbb{S}_0$, then we get the orientable surface of genus $g$ $(g\geq 0)$, denoted by $\mathbb{S}_g$. If we add $g$ crosscaps to $\mathbb{S}_0$, then we get the nonorientable surface of genus $g$ $(g\geq 1)$, denoted by $\mathbb{N}_g$. An {\it embedding} of a graph $G$ on a surface $S$ is a homeomorphism $\rho: G\rightarrow S$ of $G$ on $S$ such that every component of $S-\rho(G)$ is a 2-cell.
For a book devoted to graph embeddings, see Mohar and Thomassen~\cite{MT01}.

For a connected graph $G$, let $\gamma(G)$ and $\gamma_{M}(G)$ denote the minimum and maximum integers $g$, respectively, such that $G$ can be embedded on the orientable surface $\mathbb{S}_g$; let $\tilde{\gamma}(G)$ and $\tilde{\gamma}_{M}(G)$ denote the minimum and maximum integers $g$, respectively, such that $G$ can be embedded on the nonorientable surface $\mathbb{N}_g$. In 1966, Duke \cite{Duke1966} proved that $G$ has embeddings on all orientable surfaces $\mathbb{S}_{g}$, where $\gamma(G)\leq g\leq \gamma_{M}(G)$. This result is known as {\it Duke's Interpolation Theorem}. The nonorientable counterpart to this result is due to Stahl \cite{Stahl1978}, which stated that $G$ has embeddings on all nonorientable surfaces $\mathbb{N}_{g}$, where $\tilde{\gamma}(G)\leq g\leq \tilde{\gamma}_{M}(G)$. These two interpolation theorems are fundamental results in the theory of graph embeddings.

A {\it signed graph} $\Sigma=(G,\sigma)$ is a graph $G=(V(G), E(G))$  with a mapping $\sigma:E(G)\longrightarrow\{+1,-1\}$, which assigns each edge of $G$ a positive or negative sign.  An edge $e\in E(G)$ is called a {\it positive edge} if $\sigma(e)=1$, and a negative edge if $\sigma(e)=-1$.
A cycle of $\Sigma$ is {\it positive} if it contains an even number of negative edges, and {\it negative} otherwise.
A signed graph $\Sigma$ is called {\it balanced} if every cycle of it is positive and {\it unbalanced} otherwise.

As an extension of graphs, signed graphs have a wide range of applications, and many graph problems can be naturally formulated in the signed graphs. For an overview, see \cite{Zaslavsky1998}.  Zaslavsky~\cite{Zaslavsky1981} first posed the question of developing the embedding theory for signed graphs and later introduced signed graph embeddings~\cite{Zaslavsky1992} which was called orientation embeddings in \cite{Zaslavsky1992}.

An {\it embedding}  of a signed graph $\Sigma=(G,\sigma)$ is an embedding $\rho$ of $G$ on a surface $S$ such that a cycle $C$ of $\Sigma$ is balanced if and only if its image $\rho(C)$ is orientation-preserving on $S$.   A {\it rotation system} $R$ of  $\Sigma=(G,\sigma)$ is an assignment of a rotation at each vertex. We refer the reader to \cite{Zaslavsky1992} for a general theory of signed graph embeddings. In analogy to embeddings of ordinary graphs, several embedding parameters of signed graphs, such as demigenus and maximum Euler-genus, have been investigated in~\cite{Lv2015,Lv2018,SiranSkoviera1991,Zaslavsky1997,Zaslavsky2001}. Some other problems related to signed graph embeddings have also been studied in~\cite{Siran1991-1,Zaslavsky1993,Zaslavsky1996,Zaslavsky1997-1}.

In this paper, we focus on Euler-genus spectrum of signed graphs. Let $\Sigma=(G,\sigma)$ be a connected signed graph and $\rho$ is an embedding of $\Sigma$.
We denote $f(\rho)$ the number of faces in the embedding and $\gamma^E(\rho)$ its Euler-genus. From Euler's formula, we have
\[
\gamma^E(\rho)
=|E(G)|-|V(G)|+2-f(\rho).
\]
The \textit{Euler-genus spectrum} of $\Sigma$ is the set
\[
\mathcal{S}(\Sigma)
=\{\gamma^E(\rho):
\rho\text{ is an embedding of}~ \Sigma\}.
\]

Motivated by interpolation theorems of graphs which we mentioned above, Zaslavsky \cite{Zaslavsky1992} conjectured an analogous interpolation property for signed graphs, i.e., the Euler-genus spectrum $\mathcal{S}(\Sigma)$ is an interval of integers for every connected signed graph $\Sigma$.  It is well known  that a signed graph $\Sigma$ has an embedding on an orientable surface if and only if $\Sigma$ is balanced. Thus, Duke's interpolation theorem
extends automatically to embeddings of balanced signed graphs. Therefore, we only need to focus on unbalanced signed graphs.
Ultimately, this conjecture was disproved by \v{S}ir\'a\v{n} \cite{Siran1991}. \v{S}ir\'a\v{n} constructed a signing $\Sigma$ of $K_1+K_{3,3}$ in which the negative edges form a matching of $K_{3,3}$ and showed that
\[
\cS(\Sigma)=\{1,3,4,5,6,7,8,9\}.
\]
Thus $\cS(\Sigma)$ has a gap at $2$ and hence is not an interval.
This leads to the following weaker interpolation conjecture due to \v{S}ir\'a\v{n}. As noted by
Zaslavsky~\cite{Zaslavsky1998}, the problem was later included by Archdeacon~\cite{Archdeacon1999} in his collection of open problems in topological graph theory under the title ``The genus sequence of a signed graph''.

\begin{conjecture}[\cite{Siran1991,Archdeacon1999,Zaslavsky1998}]
\label{c1:1}
Let $\Sigma$ be a connected signed graph. If
$k,k+1\in\cS(\Sigma)$, then every integer between $k$ and the maximum
Euler-genus of $\Sigma$ also belongs to $\cS(\Sigma)$; that is,
\[
\{k,k+1,\ldots,\gamma^E_M(\Sigma)\}
\subseteq \cS(\Sigma),
\]
where $\gamma^E_M(\Sigma)=\max\cS(\Sigma).$
\end{conjecture}

We give a  characterization of the Euler-genus spectrum of a connected signed graphs as follows. This stronger result directly establishes \v{S}ir\'{a}\v{n}'s conjecture.

\begin{theorem}\label{t1:2}
Let $\Sigma$ be a connected signed graph,
$d=\min\cS(\Sigma)$ and $\gamma^E_M(\Sigma)=\max\cS(\Sigma)$. If $\tau$ is the smallest integer such that
$\tau,\tau+1\in\cS(\Sigma)$, then
\[
\cS(\Sigma)
=
\{d,d+2,\ldots,\tau\}
\cup
\{\tau+1,\tau+2,\ldots,\gamma_M^E(\Sigma)\}.
\]
\end{theorem}

To prove Theorem \ref{t1:2}, we prove the following two properties theorems about $\cS(\Sigma)$. From the following two theorems, Theorem \ref{t1:2} follows. We first introduce some notation. For $i\in\{0,1\}$, define $$\cS_i(\Sigma)=\{\gamma^E\in\cS(\Sigma):\gamma^E\equiv i\pmod2\}.$$ If $\cS_i(\Sigma)$ is nonempty, let $\gamma_{i,\min}^E=\min\cS_i(\Sigma)$ and $\gamma_{i,\max}^E=\max\cS_i(\Sigma)$.

\begin{theorem}\label{t1:3}
Let $\Sigma$ be a connected  signed graph. If both $\cS_0(\Sigma)$ and $\cS_1(\Sigma)$ are nonempty, then $\left|\gamma_{0,\max}^E-\gamma_{1,\max}^E\right|=1$.
\end{theorem}

\begin{theorem}\label{t1:4}
Let $\Sigma$ be a connected  signed graph. For $i\in\{0,1\}$, if $\cS_i(\Sigma)$ is nonempty, then
$$\cS_i(\Sigma)=\{\gamma_{i,\min}^E,\gamma_{i,\min}^E+2,\ldots,\gamma_{i,\max}^E\}.$$
\end{theorem}

The paper is organized as follows. In Section \ref{Sec:2}, we present some preliminaries needed for this paper.
In Section \ref{Sec:3}, we prove Theorem \ref{t1:3}.
In Section \ref{Sec:4}, we prove Theorem \ref{t1:4}. In Section \ref{Sec:5}, we concludes the paper with some further discussions.

\section{Preliminaries}\label{Sec:2}
In this section, we first present the basic concepts and properties of pre-signed graphs. Then we introduce adjacent-exchange operations used in later proofs. Finally, by introducing auxiliary graphs called alternating matching graphs, we discuss how the number of faces of signed graph embeddings changes under adjacent-exchange operations.

Throughout this paper, permutations are in cycle decomposition forms and the
product of permutations are read from left to right. For a permutation $P$, we denote by $\lVert P\rVert$ the number of cycles of $P$, with fixed points counted as cycles of length one.

\subsection{Pre-signed graphs}
In 1980, Stahl \cite{Stahl1980} introduced useful combinatorial objects called permutation-partition pairs which can be used in describing  graphs and orientable graph embeddings. In 1986, Archdeacon \cite{Archdeacon86-1,Archdeacon86-2} introduced identified disk spaces independently which can be viewed as geometric forms of permutation-partition pairs. By using them, some important results in topological graph theory are derived again, and some new results are also obtained, see \cite{Archdeacon86-1,Archdeacon86-2,Archdeacon98,Bonnington94-1,Bonnington94-2,Stahl1980,Stahl1982,Stahl1983,Stahl1988,Stahl1991}. In 2026, Chen \cite{Chen2026PreSigned} introduced another combinatorial objects called pre-signed graphs, which is a generalization of permutation-partition pairs. Under this generalization, nonorientable embeddings of graphs, signed graphs, and embeddings of signed graphs can be described by pre-signed graphs. In the framework of pre-signed graphs, some problems of nonorientable surface embeddings of graphs and embeddings of signed graphs have been solved \cite{Chen2026PreSigned,Chen2026PreSigned01}.
Our paper is also conducted within the framework of pre-signed graphs.

We now present the basic theory of pre-signed graphs and its correspondence with signed graphs and signed graph embeddings.
\begin{definition}
Let $F$ and $F_{\theta}$ be two disjoint finite sets with $|F|=|F_{\theta}|$, and $\mathbf F=F\cup F_{\theta}$. A \emph{pre-signed graph} is a triple $\Sigma_P=(\theta,P,\mathbf{\Pi})$, where:
\begin{enumerate}[label=(\arabic*)]
\item $\theta$ is a fixed-point-free involution on $\mathbf F$ interchanging $F$ and $F_{\theta}$, i.e., $\theta^2=\iota$ and $\theta(F)=F_{\theta}$;
\item $P$ is a permutation of $\mathbf F$ satisfying $P\theta=\theta P^{-1}$, i.e., if $(a_1\,a_2\,\ldots\,a_r)$ is a cycle of $P$, then $(\theta a_1\,\theta a_r\,\ldots\,\theta a_2)$ is also a cycle of $P$;
\item $\mathbf{\Pi}$ is a partition of $\mathbf F$ satisfies
$\mathbf{\Pi}=\Pi\cup\theta\Pi$, where $\Pi=\{\Pi_1,\ldots,\Pi_n\}$ is a partition of $F$, and $\theta\Pi=\{\theta\Pi_1,\ldots,\theta\Pi_n\}$ is a partition of $F_{\theta}$.
\end{enumerate}
The permutation $P$ is called a \emph{bi-permutation}.
The partition $\mathbf{\Pi}$ is called a \emph{bi-partition}.
For each $i$ $(1\leq i\leq n)$, $\mathbf{\Pi}_i=\Pi_i\cup\theta\Pi_i$ is called a {\it vertex} of $\Sigma_P$. For each $a\in F$, the quadruple $\{a, \theta a, Pa, \theta Pa \}$ is called an {\it edge} of $\Sigma_P$, in which $\{a, \theta a\}$ is called a {\it semi-edge} at $a$, denoted by $e_a$.
The degree of the vertex $\mathbf{\Pi}_i$ is $d(\mathbf{\Pi}_i)=|\Pi_i|=|\theta\Pi_i|$.
\end{definition}

The following proposition shows that pre-signed graphs can represent signed graphs.

\begin{proposition}[\cite{Chen2026PreSigned}]\label{p2:2}
For every signed graph $\Sigma=(G,\sigma)$, there exists a corresponding pre-signed graph.
\end{proposition}

\begin{example}  To illustrate this correspondence, consider first the simplest case in which the signed graph consists of a single edge $a=uv$. Associated with $a$ is the quadricell $K_a=\{a,\alpha a,\beta a,\epsilon a\}$, where $\epsilon=\alpha\beta=\beta\alpha$. Let $F=\{a,\beta a\}$ and $F_{\theta}=\{\alpha a,\epsilon a\}$, and let $\theta=\alpha=(a\,\alpha a)(\beta a\,\epsilon a)$. The two vertices $u$ and $v$ are represented by $\Pi_u=\{a\}$ and $\Pi_v=\{\beta a\}$ respectively, together with their $\theta$-partners $\theta\Pi_u=\{\alpha a\}$ and $\theta\Pi_v=\{\epsilon a\}$. Thus $\mathbf{\Pi}_u=\{a,\alpha a\}$, $\mathbf{\Pi}_v=\{\beta a,\epsilon a\}$ and $\mathbf{\Pi}=\{\mathbf{\Pi}_u, \mathbf{\Pi}_v\}.$

The sign of $a$ is encoded entirely in the bi-permutation $P$. If $a$ is positive, then $P=(a\,\epsilon a)(\alpha a\,\beta a)$, whereas if $a$ is negative, then $P=(a\,\beta a)(\alpha a\,\epsilon a)$. Thus the positive and negative versions of the same underlying edge have the same involution $\theta$ and the same vertex bi-partition $\mathbf{\Pi}$, and differ only in the corresponding bi-permutation $P$.

For an arbitrary signed graph $\Sigma=(G,\sigma)$, the same construction is carried out independently for every edge of $G$. The resulting bi-partition $\mathbf{\Pi}$ records the vertex incidences, while the product of the edge permutations gives $P$ and encodes both the edge structure and the signs. In this way, $\Sigma$ is represented by the pre-signed graph
$\Sigma_P=(\theta, P, \mathbf{\Pi}).$
\end{example}

We next define embeddings of a pre-signed graph.
\begin{definition} Suppose that $\Sigma_P=(\theta,P,\mathbf{\Pi})$ is a pre-signed graph and $\Pi_i=\{a_{i,1},a_{i,2},\ldots,a_{i,d_i}\}$. A \emph{bi-rotation} at the vertex $\mathbf{\Pi}_i=\Pi_i\cup\theta\Pi_i$ is a pair of cycles of the form $$(a_{i,j_1}\,a_{i,j_2}\,\cdots\,a_{i,j_{d_i}})(\theta a_{i,j_{d_i}}\,\cdots\,\theta a_{i,j_2}\,\theta a_{i,j_1}),$$ where $j_1j_2\cdots j_{d_i}$ is a permutation of $\{1,2,\ldots,d_i\}$. A \emph{bi-rotation system} $Q$ on $\mathbf{\Pi}$ is the product of one bi-rotation at each vertex. One can check that $Q$ satisfies $\theta Q=Q^{-1} \theta$. We denote by $\mathcal B(\mathbf{\Pi})$ the set of all bi-rotation systems on $\mathbf{\Pi}$.

An \emph{embedding} of the pre-signed graph $\Sigma_P=(\theta,P,\mathbf{\Pi})$ is a triple $(\theta,P,Q)$ with $Q\in\mathcal B(\mathbf{\Pi})$.
From \cite{Chen2026PreSigned}, the number of faces in an embedding $(\theta,P,Q)$ is
$f(Q)=\lVert PQ\rVert/2.$
\end{definition}

\begin{example}
Figure~\ref{fig:d3}(a) shows a signed graph $\Sigma$ with two vertices $u$ and $v$, joined by two negative edges and one positive edge, where
dashed edges are negative, the solid edge is positive. The corresponding pre-signed graph is $\Sigma_P=(\theta,P,\mathbf{\Pi})$, where
$$\theta=(1\,2)(3\,4)(5\,6)(7\,8)(9\,10)(11\,12),$$
$$P=(1\,3)(2\,4)(5\,8)(6\,7)(9\,11)(10\,12),$$
and
$$\mathbf{\Pi} = \{ \mathbf{\Pi}_u,\mathbf{\Pi}_v\}=\bigl\{ \{2,6,10\}\cup\{1,5,9\}, \{3,7,11\}\cup\{4,8,12\} \bigr\}.$$
Figure~\ref{fig:d3}(b) shows an embedding of $\Sigma$ on the projective plane, which is also an embedding of the corresponding pre-signed graph $\Sigma_P$ with the bi-rotation system
$$Q=(2\,6\,10)(9\,5\,1)(3\,7\,11)(12\,8\,4).
$$\end{example}

\begin{figure}[ht]
\begin{center}
\begin{tikzpicture}
[p/.style={circle,draw=black,fill=black,inner sep=1.4pt}]

  \node(u) at(0,0)[p]{}; \node(v) at(6,0)[p]{};
  \draw (-0.3,0) node {$u$}
        (6.3,0) node {$v$};
\draw[dashed, thick, bend left=60] (u) to (v);
\draw[dashed, thick, bend right=60] (u) to (v);
 \draw[thick] (u) -- (v);

  \node at (1, 0.8) {1};
  \node at (0.8, 1.3) {2};
  \node at (5, 0.8) {3};
  \node at (5.2, 1.3) {4};

  \node at (1.5, 0.2) {6};
  \node at (1.5, -0.2) {5};
  \node at (4.5, 0.2) {8};
  \node at (4.5, -0.2) {7};

  \node at (1, -0.8) {10};
  \node at (0.8, -1.2) {9};
  \node at (5, -0.75) {12};
  \node at (5.2, -1.2) {11};
\draw (3,-3) node {(a)};

\begin{scope}[xshift=11cm,
vertex/.style={circle,fill=black,inner sep=1.4pt},
ident/.style={->, >=stealth,line width=0.9pt}]

    \coordinate (u) at (0,0.9);
    \coordinate (v) at (0,-0.9);

    \coordinate (A) at (-1.7,1.62);
    \coordinate (B) at (1.7,1.62);
    \coordinate (C) at (1.7,-1.62);
    \coordinate (D) at (-1.7,-1.62);

    \draw(0,0) ellipse (3.0 and 2.0);

    \draw[thick] (u) -- (A);
    \draw[thick] (u) -- (B);
    \draw[thick] (u) -- (v);
    \draw[thick] (v) -- (C);
    \draw[thick] (v) -- (D);

    \node[vertex] at (u) {};
    \node[vertex] at (v) {};

    \node[above=2pt] at (u) {$u$};
    \node[below=2pt] at (v) {$v$};

    \draw[ident]
        (1.45,1.98)
        .. controls (0.70,2.23) and (-0.70,2.23) ..
        (-1.45,1.98);

    \draw[ident]
        (-1.45,-1.98)
        .. controls (-0.70,-2.23) and (0.70,-2.23) ..
        (1.45,-1.98);

    \node at (-1.05,1.62) {$9$};
    \node at (-1.05,1.05) {$10$};

    \node at (1.05,1.62) {$2$};
    \node at (1.05,1.05) {$1$};

    \node at (-0.27,0.52) {$5$};
    \node at (0.27,0.52) {$6$};

    \node at (-0.27,-0.52) {$7$};
    \node at (0.27,-0.52) {$8$};

    \node at (-1.05,-1.05) {$4$};
    \node at (-1.05,-1.62) {$3$};

    \node at (1.05,-1.05) {$11$};
    \node at (1.05,-1.62) {$12$};
\draw (0,-3) node {(b)};
\end{scope}
\end{tikzpicture}
 \caption{A signed graph and an embedding of it.}
\label{fig:d3}
\end{center}
\end{figure}
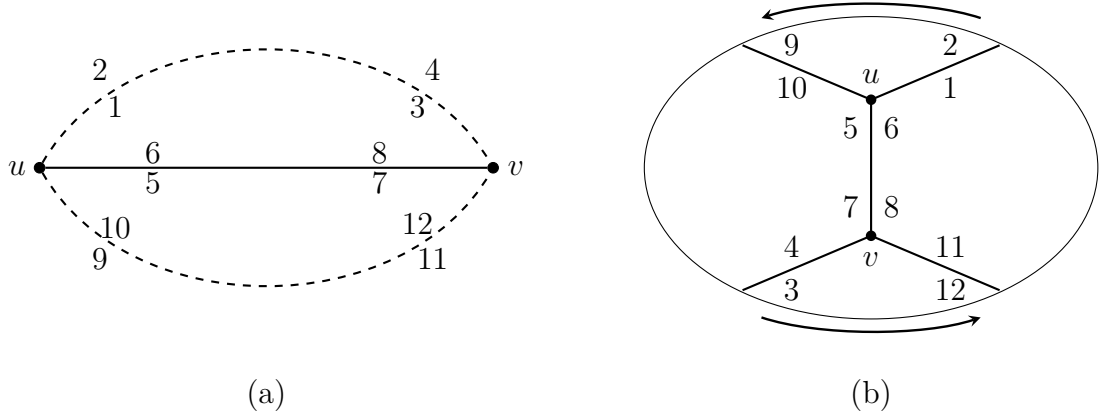

For more examples illustrating the relationship between embeddings of signed graphs and embeddings of pre-signed graphs, we refer to \cite{Chen2026PreSigned,Chen2026PreSigned01}.

By Proposition \ref{p2:2}, a signed graph can be represented by a pre-signed graph. Furthermore, there is a natural one-to-one correspondence between the rotation systems of a signed graph and bi-rotation systems of the corresponding pre-signed graph, so the embedding problems for signed graphs can be transformed into the embedding problems for pre-signed graphs. Therefore, in the following, we will study the embeddings in the framework of pre-signed graphs.

\subsection{Adjacent-exchange operations}
In this subsection, we will introduce an operation called adjacent-exchange that can change one bi-rotation system in an embedding into another bi-rotation system.

For an embedding of a pre-signed graph $\Sigma_P$ with a bi-rotation system $Q_0$, suppose that the bi-rotation at the vertex $v$ is
$$Q_{0,v}=(a_{v,1}\ldots a_{v,i} a_{v,i+1}\ldots a_{v,d_v}) (\theta a_{v,d_v}\ldots\theta a_{v,i+1}  \theta a_{v,i}\ldots\theta a_{v,1}).$$
For $1\leq i\leq d_{v}-1$, an {\it adjacent-exchange} on the bi-rotation at vertex $v$ exchanges $a_{v,i}$ and $a_{v,i+1}$  in the first cycle, and simultaneously exchanges  $\theta a_{v,i}$ and
$\theta a_{v,i+1}$ in the second cycle.
By the adjacent-exchange operation, we get a new bi-rotation system $Q_1$ in which the  bi-rotation at vertex $v$ is
$$Q_{1,v}=(a_{v,1}\ldots a_{v,i+1} a_{v,i}\ldots a_{v,d_v}) (\theta a_{v,d_v}\ldots\theta a_{v,i}\theta a_{v,i+1}\ldots\theta a_{v,1}).$$
This operation can be realized by conjugating the bi-rotation with a permutation. We define
$$
\delta_i=
(a_{v,i}\,a_{v,i+1})
(\theta a_{v,i}\,\theta a_{v,i+1}),
$$and call it  a {\it bi-transposition}, then
$$Q_{1,v}=\delta_i Q_{0,v} \delta_i^-.$$

For any vertex $v$ with initial bi-rotation $Q_{0,v}$, let $\ell_v$ denote the number of  adjacent-exchange operations at $v$, and let $k_{v,j}\in\{1,2,\ldots,d_v-1\}$ $(1\leq j \leq \ell_v)$ be their exchange positions.
We denote by $p_{v,j}$ the $j$-th adjacent-exchange operation at vertex $v$, occurring at position $k_{v,j}$. It exchanges the elements at positions $k_{v,j}$ and $k_{v,j}+1$ in each of the two cycles of the current bi-rotation. Let
$$\mathcal P=\{p_{v,j}: v ~\mbox{is a vertex of}~ \Sigma_P, 1\leq j\leq \ell_v\}.$$ and $m=\sum_{v\in V}\ell_v$.

For simplicity, we denote
$\mathcal P=\{p_1,p_2,\ldots,p_m\}$, and equip them with a total order $p_1\prec p_2\prec\cdots\prec p_m$
that respects the prescribed order at each vertex.
Since operations at distinct vertices act on disjoint sets of symbols, the order in which they are interleaved does not matter. If $p_i=p_{v,j}$, then the corresponding  bi-transposition is
$$
\delta_i=
(a_{v,k_{v,j}}\,a_{v,k_{v,j}+1})
(\theta a_{v,k_{v,j}}\,\theta a_{v,k_{v,j}+1}),
$$
with respect to the initial bi-rotation system $Q_0$. In the following, we will always use $\delta_i$ to denote the bi-transposition corresponding to $p_i$.

For $A\subseteq\mathcal P$, we let
$$\chi_A(i)= \begin{cases} 1, & p_i\in A,\\ 0, & p_i\notin A, \end{cases} \qquad 1\le i\le m,$$  and
$$g_A=
\delta_1^{\chi_A(1)}
\delta_2^{\chi_A(2)}
\cdots
\delta_m^{\chi_A(m)}.$$
For $0\leq i\leq m$,  we denote
$g_A^{(i)}=\delta_1^{\chi_A(1)}\cdots
\delta_i^{\chi_A(i)}$, in which $g_A^{(0)}=\iota$ and
$g_A^{(m)}=g_A$.

Let $Q_A$ be the bi-rotation system obtained from $Q_0$ by performing adjacent-exchange operations in $A$ in the prescribed order. For
$0\leq i\leq m$, let $Q_A^{(i)}$ denote the bi-rotation system obtained
after the first $i$ exchange operations have been processed. Thus
$Q_A^{(0)}=Q_0$ and $Q_A^{(m)}=Q_A$. If for some $i$ $(1\leq i \leq m)$, $p_i\notin A$, then we do nothing in this step and $Q_A^{(i)}=Q_A^{(i-1)}$ follows. We also note that $g_A^{(i-1)}$
can be seen as a combined effect of the first $i-1$ exchange operations in $A$.  When the first $i-1$ exchange operations have already been applied, the action of $\delta_i$ on the current bi-rotation system $Q_{A}^{(i-1)}$
is given by the conjugate
\[
\widehat{\delta}_i
=
g_A^{(i-1)}\delta_i
\bigl(g_A^{(i-1)}\bigr)^{-1}.
\]

\begin{example}Suppose that the bi-rotation at a vertex $v$ is
$$(a_1,a_2,a_3,a_4) (\theta a_4,\theta a_3,\theta a_2,\theta a_1),$$
and that the sequence of exchange positions is
$(k_{v,1},k_{v,2},k_{v,3})=(1,2,1).$
The corresponding bi-transpositions are
$$\begin{aligned}
\delta_1&=(a_1\,a_2)(\theta a_1\,\theta a_2),\\
\delta_2&=(a_2\,a_3)(\theta a_2\,\theta a_3),\\
\delta_3&=(a_1\,a_2)(\theta a_1\,\theta a_2).
\end{aligned}$$
Thus, $$\begin{aligned}
\widehat\delta_1 &=\delta_1 =(a_1\,a_2)(\theta a_1\,\theta a_2),\\
\widehat\delta_2 &=\delta_1\delta_2\delta_1^{-1} =(a_1\,a_3)(\theta a_1\,\theta a_3),\\
\widehat\delta_3 &=(\delta_1\delta_2)\delta_3 (\delta_1\delta_2)^{-1} =(a_2\,a_3)(\theta a_2\,\theta a_3).
\end{aligned}$$
Applying these three ordered adjacent exchanges successively yields the following sequence of bi-rotations:
$$\begin{aligned}
&(a_1,a_2,a_3,a_4) (\theta a_4,\theta a_3,\theta a_2,\theta a_1)\\
&\longrightarrow (a_2,a_1,a_3,a_4) (\theta a_4,\theta a_3,\theta a_1,\theta a_2)\\
&\longrightarrow (a_2,a_3,a_1,a_4) (\theta a_4,\theta a_1,\theta a_3,\theta a_2)\\
&\longrightarrow (a_3,a_2,a_1,a_4) (\theta a_4,\theta a_1,\theta a_2,\theta a_3).
\end{aligned}$$
\end{example}

The following lemma shows that any two bi-rotation systems in $\mathcal B(\mathbf{\Pi})$ can be connected by a finite sequence of adjacent-exchange operations, and these operations can be realized by conjugation.

\begin{lemma} \label{lem2:7} For every $A\subseteq\mathcal P$, $Q_A=g_AQ_0g_A^{-1}$.\end{lemma}

\begin{proof} We first prove  $Q_A^{(i)}=g_A^{(i)}Q_0(g_A^{(i)})^{-1}$ by induction. For $i=0$, we have $Q_A^{(0)}=Q_0$. Since $g_A^{(0)}=\iota$, we have $Q_A^{(0)}=g_A^{(0)}Q_0(g_A^{(0)})^{-1}$.

Assume that, after the first $i-1$ exchange operations have been processed, $Q_A^{(i-1)}=g_A^{(i-1)}Q_0(g_A^{(i-1)})^{-1}$.   For the $i$th exchange operation $p_i$ with  bi-transposition $\delta_i$, the bi-transposition that acts on
$Q_A^{(i-1)}$ is $\widehat\delta_i=g_A^{(i-1)}\delta_i(g_A^{(i-1)})^{-1}$. If $p_i\in A$, then $$Q_A^{(i)} =\widehat\delta_iQ_A^{(i-1)}\widehat\delta_i^{-1} =g_A^{(i-1)}\delta_iQ_0 (g_A^{(i-1)}\delta_i)^{-1} =g_A^{(i)}Q_0(g_A^{(i)})^{-1}.$$ If $p_i\notin A$, then $Q_A^{(i)}=Q_A^{(i-1)}$ and $g_A^{(i)}=g_A^{(i-1)}$, so $Q_A^{(i)}=g_A^{(i)}Q_0(g_A^{(i)})^{-1}$ holds. Thus the induction is complete.

Then, taking $i=m$, we get $Q_A=Q_A^{(m)}$ and $g_A^{(m)}=g_A$, $Q_A=g_AQ_0g_A^{-1}$ follows. The proof is complete.
 \end{proof}

For $A,B\subseteq\mathcal P$, their \emph{symmetric difference} is defined by
$A\triangle B=(A\setminus B)\cup(B\setminus A)$. Thus
$A\triangle B$ consists precisely of the exchange operations that belong
to exactly one of $A$ and $B$. Since $\mathcal P$ is equipped with a
fixed total order, we regard $A\triangle B$ as inheriting this order.
In particular, if
$A\triangle B=\{p_{j_1},\ldots,p_{j_r}\}$, then its elements are always
listed so that $j_1\prec\cdots\prec j_r$.
From Lemma \ref{lem2:7}, we have $Q_A=g_AQ_0g_A^{-1}$ and $Q_B=g_BQ_0g_B^{-1}$. The following lemma shows that relation between $g_A$ and $g_B$ is related to $A\triangle B$.

\begin{lemma}
\label{lem2:8}
For $A,B\subseteq\mathcal P$,
if $A\triangle B=\{p_{j_1},\ldots,p_{j_r}\}$, then
$$g_Ag_B^{-1}
=
\widehat\delta_{j_1}\cdots\widehat\delta_{j_r},$$
where
$\widehat\delta_j
=g_B^{(j-1)}\delta_j\bigl(g_B^{(j-1)}\bigr)^{-1}$.
\end{lemma}
\begin{proof}
For $0\le j\le m$, we define
$h_j=g_A^{(j)}\bigl(g_B^{(j)}\bigr)^{-1}$.
We first prove
\begin{equation}
h_j=
\prod_{\substack{k\le j,\ p_k\in A\triangle B}}
\widehat\delta_k
\label{eq1}
\end{equation}
by induction, where the factors are taken in increasing order of $k$, and the empty product is understood to be $\iota$.

For $j=0$, we have $h_0=\iota$, so \eqref{eq1} holds. Suppose that \eqref{eq1} holds for $j-1$. Since
$g_A^{(j)}=g_A^{(j-1)}\delta_j^{\chi_A(j)}$,
$g_B^{(j)}=g_B^{(j-1)}\delta_j^{\chi_B(j)}$,
and $\delta_j^{-1}=\delta_j$, we have
$$h_j=g_A^{(j-1)}
\delta_j^{\chi_A(j)}
\delta_j^{\chi_B(j)}
\bigl(g_B^{(j-1)}\bigr)^{-1}.$$

If $\chi_A(j)=\chi_B(j)$, then
$\delta_j^{\chi_A(j)}\delta_j^{\chi_B(j)}=\iota$, and
$h_j=h_{j-1}$. If $\chi_A(j)\ne\chi_B(j)$, then
$\delta_j^{\chi_A(j)}\delta_j^{\chi_B(j)}=\delta_j$, and
$$
h_j
=
g_A^{(j-1)}\delta_j
\bigl(g_B^{(j-1)}\bigr)^{-1}
=
h_{j-1}
g_B^{(j-1)}\delta_j
\bigl(g_B^{(j-1)}\bigr)^{-1}
=
h_{j-1}\widehat\delta_j.
$$
Thus, at step $j$, no new factor is introduced when
$p_j\notin A\triangle B$, whereas $\widehat\delta_j$ is appended when
$p_j\in A\triangle B$. This proves \eqref{eq1} by induction.

Then, taking $j=m$, we obtain
$g_Ag_B^{-1}=h_m
=\widehat\delta_{j_1}\cdots\widehat\delta_{j_r}$, the proof is complete.
\end{proof}

\subsection{Transpositions and Cycle Count}
In this subsection, we introduce an auxiliary graph defined by two fixed-point-free involutions. Later we will use it to count cycles, and thereby count the number of faces in embeddings.

Let $\mu$ and $\nu$ be fixed-point-free involutions on a finite set
$\Omega$. Then their cycle decompositions consist entirely of disjoint $2$-cycles .
Each such $2$-cycle of $\mu$ pairs two elements $x$ and $\mu(x)$;
each $2$-cycle of $\nu$ pairs two elements $x$ and $\nu(x)$.

We define
\[
M_\mu := \{ x\mu(x) : x \in \Omega \}, ~~\mbox{and}~~M_\nu := \{ x\nu(x) : x \in \Omega \},
\]
and called them {\it perfect matchings on $\Omega$}. Futhermore, we define the {\it alternating matching graph} $\mathcal H(\mu,\nu)$, by taking the vertex set $V(\mathcal H(\mu,\nu))=\Omega$ and the edge set
\[
E(\mathcal H(\mu,\nu)) = M_\mu \sqcup M_\nu,
\]
where the symbol $\sqcup$ denotes a disjoint union: if an edge belongs to both $M_\mu$ and $M_\nu$, we keep both two copies.  We call edges in $M_\mu$ and $M_\nu$ are $\mu$-edges and $\nu$-edges, respectively.

From definitions of $\mu$, $\nu$ and $\mathcal H(\mu,\nu)$,  every vertex $v\in V(\mathcal H(\mu,\nu))$ is incident with exactly one $\mu$-edge and exactly one $\nu$-edge, so every vertex has degree $2$. Moreover, as we traverse a component of the graph, the edges alternate between $\mu$-edges and $\nu$-edges. We call each component of
$\mathcal H(\mu,\nu)$ is an {\it alternating cycle}. Since it is a cycle, it is a bipartite graph. Let $c(\mu,\nu)$ be the number of components of $\mathcal H(\mu,\nu)$.
It is easy to check that
\[
c(\mu,\nu)=\frac12 \|\mu\nu\|.
\]

\begin{example}
On set $\Omega=\{1,\ldots, 12\}$, we let $\mu=(1,4)(2,3)(5,7)(6,8)(9,12)(10,11)$ and $\nu=(1,6)(2,9)(3,12)(4,7)(5,10)(8,11)$ be two fixed-point-free involutions. Then $\mu\nu=(1,7,10,8)(2,12)(3,9)(4,6,11,5)$.
Figure~\ref{fig:AM} illustrates the alternating matching graph $\mathcal H(\mu,\nu)$, the edges of $M_\mu$ are shown as dashed lines, while those of $M_{\nu}$ are shown as  solid lines. And we use solid dots and hollow dots to represent the bipartition of the vertices in each component. The graph $\mathcal H(\mu,\nu)$ has
$c(\mu,\nu)=\frac12\lVert\mu\nu\rVert=2$ components: one is the alternating $8$-cycle $(1,4,7,5,10,11,8,6)$ and the other is the alternating $4$-cycle $(2,3,12,9)$.
\end{example}

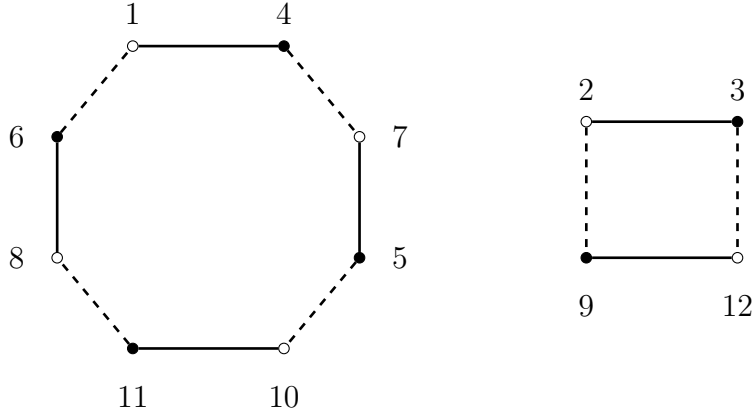
\begin{figure}[ht]
 \begin{center}
 \begin{tikzpicture}[
    vertex/.style={circle, draw=black, fill=white, inner sep=1.4pt},
    blackvertex/.style={circle, draw=black, fill=black, inner sep=1.4pt},
    solidedge/.style={line width=1pt},
    dashededge/.style={dashed, line width=1pt}
]

\node[vertex, label={[yshift=2pt]90:1}]  (v1)  at (0,2.8) {};
\node[blackvertex, label={[yshift=2pt]90:4}] (v4) at (2.0,2.8) {};

\node[vertex, label={[xshift=6pt]0:7}]  (v7) at (3.0,1.6) {};
\node[blackvertex, label={[xshift=6pt]0:5}] (v5) at (3.0,0.0) {};

\node[vertex, label={[yshift=-8pt]270:10}] (v10) at (2.0,-1.2) {};
\node[blackvertex, label={[yshift=-8pt]270:11}] (v11) at (0,-1.2) {};

\node[vertex, label={[xshift=-6pt]180:8}] (v8) at (-1.0,0.0) {};
\node[blackvertex, label={[xshift=-6pt]180:6}] (v6) at (-1.0,1.6) {};

\draw[solidedge] (v1) -- (v4);
\draw[solidedge] (v7) -- (v5);
\draw[solidedge] (v10) -- (v11);
\draw[solidedge] (v8) -- (v6);

\draw[dashededge] (v4) -- (v7);
\draw[dashededge] (v5) -- (v10);
\draw[dashededge] (v11) -- (v8);
\draw[dashededge] (v6) -- (v1);

\node[vertex, label={[yshift=2pt]90:2}] (v2) at (6.0,1.8) {};
\node[blackvertex, label={[yshift=2pt]90:3}] (v3) at (8.0,1.8) {};

\node[blackvertex, label={[yshift=-8pt]270:9}] (v9) at (6.0,0) {};
\node[vertex, label={[yshift=-8pt]270:12}] (v12) at (8.0,0) {};

\draw[solidedge] (v2) -- (v3);
\draw[solidedge] (v9) -- (v12);

\draw[dashededge] (v2) -- (v9);
\draw[dashededge] (v3) -- (v12);

\end{tikzpicture}
\end{center}
 \caption{The graph $\mathcal H(\mu,\nu)$, where
$\mu=(1,4)(2,3)(5,7)(6,8)(9,12)(10,11)$ and
$\nu=(1,6)(2,9)(3,12)(4,7)(5,10)(8,11).$}
  \label{fig:AM}
\end{figure}

The following result, due to Stahl, concerns the effect of a transposition on the number of cycles of a permutation. We will apply this result to our proof.

\begin{lemma}[\cite{Stahl1987}] \label{lem2:10}
Let $\pi$ be a permutation of a finite set $\Omega$, and let $(ab)$ be a transposition. Then the following hold:
\begin{enumerate}
    \item If $a$ and $b$ lie in the same cycle of $\pi$, then $\lVert (ab)\pi\rVert=\lVert \pi\rVert+1$, and $a$ and $b$ lie in distinct cycles of $(ab)\pi$.
    \item If $a$ and $b$ lie in distinct cycles of $\pi$, then $\lVert (ab)\pi\rVert=\lVert \pi\rVert-1$, and $a$ and $b$ lie in the same cycle of $(ab)\pi$.
\end{enumerate}
\end{lemma}

\begin{lemma}
\label{lem2:11}
Let $\mu$ and $\nu$ be fixed-point-free involutions on a finite set
$\Omega$, and $(ab)$ be a transposition. If $\nu'=(ab)\nu (ab),$ then $|c(\mu,\nu')-c(\mu,\nu)|\le 1.$
More precisely,
\[
c(\mu,\nu')-c(\mu,\nu)=
\begin{cases}
-1, & \text{if $a,b$ lie in distinct components of }
       \mathcal H(\mu,\nu),\\
+1, & \text{if $a,b$ lie in the same component and in the same}\\
    & \qquad \text{bipartition class of }\mathcal H(\mu,\nu),\\
0,  & \text{if $a,b$ lie in the same component and in opposite}\\
    & \qquad \text{bipartition classes of }\mathcal H(\mu,\nu).
\end{cases}\]
\end{lemma}

\begin{proof}
Let $P=\mu\nu$, $a'=a\nu$, and $b'=b\nu$.  Then
$(a'b')=\nu(ab)\nu$, equivalently, $(ab)\nu=\nu (a'b')$. Thus,
$$\mu\nu'=\mu(ab)\nu(ab)=\mu\nu (a'b')(ab)=P(a'b')(ab).$$
Since $P(a'b')(ab)$ and $(a'b')(ab)P$ are conjugate,
$\lVert\mu\nu'\rVert=\lVert (a'b')(ab)P\rVert$.

Recall that every component of $\mathcal H(\mu,\nu)$ is an alternating even cycle.
On each component, applying $P=\mu\nu$ means following first a $\mu$-edge and then a $\nu$-edge, so the two
bipartition classes of the component are precisely the two cycles of
$P$. Moreover, $\nu$ interchanges these two classes. Thus, if $a$ lies
in one of these cycles, then $a'=a\nu$ lies in the other.

\medskip
\noindent\textbf{Claim 1.}
If $a$ and $b$ lie in distinct components of
$\mathcal H(\mu,\nu)$, then $a'$ and $b'$ lie in distinct cycles of
$(ab)P$.
\smallskip

\noindent\emph{Proof of the claim.}
The component of $\mathcal H(\mu,\nu)$ containing $a$ corresponds to
two partner cycles of $P$, one containing $a$ and the other containing
$a'=a\nu$. Similarly, the component containing $b$ corresponds to two
partner cycles of $P$, one containing $b$ and the other containing
$b'=b\nu$. Since $a$ and $b$ lie in distinct components, these four cycles are distinct. The transposition $(ab)$ merges the two
cycles containing $a$ and $b$, while the two partner cycles containing
$a'$ and $b'$ are unchanged. Hence $a'$ and $b'$ lie in distinct cycles of $(ab)P$. \qed

If $a$ and $b$ lie in distinct components of
$\mathcal H(\mu,\nu)$, then they lie in distinct cycles of $P$. By
Lemma~\ref{lem2:10},
$\lVert(ab)P\rVert=\lVert P\rVert-1$.
By Claim 1, $a'$ and $b'$ lie in distinct cycles of $(ab)P$. Applying
Lemma~\ref{lem2:10} again, we obtain
$\lVert (a'b')(ab)P\rVert=\lVert(ab)P\rVert-1=\lVert P\rVert-2$.

\medskip
\noindent\textbf{Claim 2.}
If $a$ and $b$ lie in the same component of $\mathcal H(\mu,\nu)$ and in the same bipartition class, then $a'$ and $b'$ lie in the same cycle of $(ab)P$.

\smallskip
\noindent\emph{Proof of the claim.}
Since $a$ and $b$ lie in the same bipartition class, they lie in the same cycle of $P$. The vertices $a'=a\nu$ and $b'=b\nu$ lie in the other cycle of $P$, but in the same component of $\mathcal H(\mu,\nu)$ as $a$ and $b$. The transposition $(ab)$ splits the cycle containing $a$ and $b$, while leaving the other cycle unchanged. Hence $a'$ and $b'$ remain in the same cycle of $(ab)P$. \qed

\medskip
If $a$ and $b$ lie in the same component and in the same bipartition class, then they lie in the same cycle of $P$. By Lemma~\ref{lem2:10},
$\lVert(ab)P\rVert=\lVert P\rVert+1$.
By Claim 2 and Lemma~\ref{lem2:10}, we obtain
$\lVert (a'b')(ab)P\rVert=\lVert(ab)P\rVert+1=\lVert P\rVert+2$.

\medskip
\noindent\textbf{Claim 3.}
If $a$ and $b$ lie in the same component of $\mathcal H(\mu,\nu)$ but in opposite bipartition classes, then $a'$ and $b'$ lie in the same cycle of $(ab)P$.

\smallskip
\noindent\emph{Proof of the claim.}
Since $a$ and $b$ lie in opposite bipartition classes, they lie in the two distinct cycles of $P$ associated with the same component. Moreover, $a'=a\nu$ and $b'=b\nu$ lie in the same cycles as
$b$ and $a$, respectively. The transposition $(ab)$ merges these two cycles into a single cycle. Hence $a'$ and $b'$ lie in the same cycle of $(ab)P$.  \qed

\medskip
If $a$ and $b$ lie in the same component but in opposite bipartition classes, then they lie in distinct cycles of $P$. By Lemma~\ref{lem2:10},
$\lVert(ab)P\rVert=\lVert P\rVert-1$.
By  Claim 3 and Lemma~\ref{lem2:10}, we have $\lVert (a'b')(ab)P\rVert=\lVert P\rVert$.

Since $c(\mu,\nu)=\frac12\lVert\mu\nu\rVert$ and
$c(\mu,\nu')=\frac12\lVert\mu\nu'\rVert$, the three cases give
$c(\mu,\nu')-c(\mu,\nu)=-1,+1,0$, respectively.  The Lemma follows.
\end{proof}

\begin{corollary}
\label{cor2:12}
Let $\mu$ and $\nu$ be fixed-point-free involutions on a finite set $\Omega$, and let $\delta=ts$ where $t$ and $s$ are disjoint transpositions. Then
$\bigl|c(\mu,\delta\nu\delta^{-1})-c(\mu,\nu)\bigr|\leq 2$.
\end{corollary}

\begin{proof}
Since $t$ and $s$ are disjoint transpositions, they are commuting involutions. Hence
$\delta\nu\delta^{-1}=t(s\nu s)t$.
By Lemma~\ref{lem2:11}, conjugation by each of $s$ and $t$ changes the number of components by at most $1$. By the triangle inequality,
$\bigl|c(\mu,\delta\nu\delta^{-1})-c(\mu,\nu)\bigr|\leq 2$. The proof is complete.
\end{proof}

\subsection{Face Count}
 Now we consider any embedding $(\theta,P,Q)$ of the pre-signed graph $\Sigma_P=(\theta,P,\mathbf{\Pi})$ with $Q\in\mathcal B(\mathbf{\Pi})$. We define
$$ \mu_P:=P\theta,
  \qquad
  \nu_Q:=\theta Q.$$
From the definition of the pre-signed graph, we have $P\theta=\theta P^{-1}$. Thus
\[
(P\theta)^2=P(\theta P\theta)=PP^{-1}=\iota.
\]Moreover, $a$ and $\theta a$ belong to distinct cycles of $P$,
so $P(\theta a)\neq a$ for every $a$. Hence $P\theta$ is
fixed-point-free, and therefore $\mu_P$ is a
fixed-point-free involution. Similarly, from the definition of the bi-rotation system, we can obtain that $\nu_Q$ is also a
fixed-point-free involution.

\begin{lemma}\label{lem2:13}
If $\nu_0=\theta Q_0$, $\nu_A=\theta Q_A$ and $\nu_B=\theta Q_B$, then
$$\nu_A=g_A\nu_0g_A^{-1},~~~~\mbox{and}~~~~\nu_A=(g_Ag_B^{-1})\nu_B(g_Ag_B^{-1})^{-1}.$$
\end{lemma}
\begin{proof}
Since every bi-transposition $\delta_i$ commutes with $\theta$, so do $g_A$ and $g_B$. Combining this with Lemma \ref{lem2:7}, we have
$$\nu_A=\theta g_AQ_0g_A^{-1} =g_A(\theta Q_0)g_A^{-1} =g_A\nu_0g_A^{-1},$$and
$$\nu_B=\theta g_BQ_0g_B^{-1} =g_B(\theta Q_0)g_B^{-1} =g_B\nu_0g_B^{-1}.$$

Then $$\nu_A=(g_Ag_B^{-1})(g_B\nu_0g_B^{-1})(g_Ag_B^{-1})^{-1} =(g_Ag_B^{-1})\nu_B(g_Ag_B^{-1})^{-1}.$$
The proof is complete.
\end{proof}

\begin{lemma}\label{lem2:14} The number of faces in any embedding $(\theta,P,Q)$ of a pre-signed graph $\Sigma_P$ is $$f(Q)=c(\mu_P,\nu_Q).$$
\end{lemma}
\begin{proof}
From Proposition~6 in ~\cite{Chen2026PreSigned}, we have
$\lVert PQ\rVert=2f(Q).$ On the other hand, $$2c(\mu_P,\nu_Q)=\lVert\mu_P\nu_Q\rVert= \lVert(P\theta)(\theta Q)\rVert=\lVert PQ\rVert.$$ The lemma follows.
\end{proof}

 \begin{corollary}
\label{cor2:15}
If $A\triangle A'\subseteq\mathcal P$ has exactly one element, then
$|f(Q_A)-f(Q_{A'})|\leq 2$.
\end{corollary}

\begin{proof}
We assume $A$ and $A'$ differ only at $p_j$, and further assume $p_j\notin A$ and $p_j\in A'$.
Let $g_A=rs$ and $g_{A'}=r\delta_j s$, where $r$ is the product of the corresponding bi-transpositions
before $p_j$ and $s$ is the product of the corresponding bi-transpositions after
$p_j$. From Lemma \ref{lem2:8}, we have
$$g_{A'}g_A^{-1}=\widehat\delta_j=r\delta_jr^{-1}.$$
Let $h=r\delta_jr^{-1}$, since $\delta_j$ is a bi-transposition, so is
$h$. By Lemma~\ref{lem2:13},
$\nu_{A'}=h\nu_Ah^{-1}$. Then, from
Corollary~\ref{cor2:12},
$\bigl|c(\mu,\nu_{A'})-c(\mu,\nu_A)\bigr|\leq2$.
Combining this with Lemma~\ref{lem2:14},
$|f(Q_{A'})-f(Q_A)|\leq2$ follows.
\end{proof}

For $i\in\{0,1\}$, we define
$\mathcal F_i = \{f(Q):Q\in B(\Pi),\ f(Q)\equiv i\pmod 2\}.$
For each nonempty set $\mathcal F_i$, we denote its minimum and maximum elements by
$m_i=\min\mathcal F_i$ and $M_i=\max\mathcal F_i$, respectively.

\begin{remark}\label{re2:16}
Given a connected signed graph $\Sigma$, by Euler's formula, the Euler-genus of an embedding depends on its number of faces, so the Euler-genus spectrum problem is equivalent to the problem of the possible numbers of faces of all its embeddings. More precisely, Theorem \ref{t1:3} is equivalent to:
\begin{center}
(a) If both $\mathcal F_0$ and $\mathcal F_1$ are nonempty, then $|m_0-m_1|=1.$
\end{center}
And Theorem \ref{t1:4} is equivalent to:
\begin{center}
(b) For $i\in\{0,1\}$, if $\mathcal F_i$ is nonempty, then $\mathcal F_i=\{m_i, m_i+2, \ldots, M_i\}$.\end{center}
In the following two sections, we will prove Theorems \ref{t1:3} and \ref{t1:4} by studying $\mathcal F_i$.
\end{remark}

\section{The proof of  Theorem \ref{t1:3}}\label{Sec:3}
In this section, we will prove Theorem \ref{t1:3}. We use the same notation as in Section 2. For simplicity, we simplify some of the notation. We abbreviate $\mu_P$ as $\mu$, since $P$ does not change herein. For any bi-rotation system $Q_A$ obtained from the initial bi-rotation system $Q_0$ by adjacent-exchange operations in $A$, $A\subseteq\mathcal P$, we identify the ordered set $\mathcal P=\{p_1,p_2,\ldots,p_m\}$ with $\{1,2,\ldots,m\}$ via $p_i\leftrightarrow i$,
and correspondingly, elements in
$A$ are also denoted by subscripts.
We also abbreviate $c(\mu,\nu_A)$ as $c(A)$, $c(\mu,\nu_0)$ as $c(\varnothing)$.

The following lemma is the main tool in this section.

\begin{lemma}
\label{lem3:1}
For $A\subseteq \mathcal P$, we assume $c(\varnothing)=\min_A c(A)$. If $c(A)-c(\varnothing)$ is odd for some $A$, then there exists $A_0\subseteq\mathcal P$ such that $c(A_0)=c(\varnothing)+1$.
\end{lemma}

\begin{proof}
We prove by induction on $m$, the number of adjacent-exchange operations in $\mathcal P$.
Each adjacent-exchange operation corresponds to
a bi-transposition $\delta_i$ which is a product of two disjoint transpositions.
So we let $\delta_i=t_is_i$, where $t_i$ and $s_i$ are disjoint transpositions, for $1\leq i\leq m$.

For $m=1$, the only subsets are $\varnothing$ and $\{1\}$. Since $c(\varnothing)$ is minimal, and by Corollary~\ref{cor2:12}, a bi-transposition changes $c$ by at most $2$, it follows that $c(\{1\})-c(\varnothing)\in\{0,1,2\}$. If this difference is odd, then we have $c(\{1\})-c(\varnothing)=1$. We now take $A_0=\{1\}$, the result holds for this case.

Now suppose that $m\geq2$ and that the result holds for $m-1.$
Again by
Corollary~\ref{cor2:12}, for every $i$ we have
$c(\{i\})-c(\varnothing)\in\{0,1,2\}$.
If $c(\{i\})=c(\varnothing)+1$ for some $i$, then we may take
$A_0=\{i\}$. Thus, we may assume that no singleton has $c(\{i\})=c(\varnothing)+1$.

\medskip
\noindent\textbf{Claim.}
There exists an index $i\in\{1,\ldots,m\}$ such that
$c(\{i\})=c(\varnothing)$.

\smallskip
\noindent\textit{Proof of the claim.}
Suppose for contradiction that
\begin{equation}
c(\{i\})=c(\varnothing)+2
\qquad (1\leq i\leq m).
\label{eq2}
\end{equation}

Since $\mathcal H(\mu,\nu_0)$ is bipartite, choose a bipartition
$(U,V)$ of its vertex set. For each $i$, conjugation by
$\delta_i=t_is_i$ may be performed as two successive conjugations,
first by $t_i$ and then by $s_i$, since
$\delta_i\nu_0\delta_i^{-1}=s_i(t_i\nu_0t_i)s_i$.
By \eqref{eq2} and
Lemma~\ref{lem2:11}, each of the two conjugations must
increase the number of components by exactly $1$.

Applying Lemma~\ref{lem2:11} to the first
conjugation, the two entries of $t_i$ lie in the same bipartition
class. Hence $t_i(U)=U$ and $t_i(V)=V$. It follows that
$t_i\nu_0t_i$ still pairs $U$ with $V$, so $(U,V)$ remains a
bipartition of $\mathcal H(\mu,t_i\nu_0t_i)$.
Applying Lemma~\ref{lem2:11} to the second
conjugation, the two entries of $s_i$ also lie in the same
bipartition class. Hence $s_i(U)=U$ and $s_i(V)=V$.
Therefore $\delta_i(U)=U$ and $\delta_i(V)=V$ for every $i$, and
consequently $g_A(U)=U$ and $g_A(V)=V$ for every $A$.
Thus every $\nu_A$ is a perfect matching between $U$ and $V$.

Let $\phi:U\to V$ and $\psi_A:U\to V$ be the bijections induced by
$\mu$ and $\nu_A$, respectively, and let
$\pi_A=\psi_A\phi^{-1}$. Then $\pi_A$ is a permutation of $U$.
Applying $\pi_A$ means first following a $\nu_A$-edge from $U$ to
$V$ and then a $\mu$-edge from $V$ back to $U$. Hence each cycle of
$\pi_A$ corresponds to exactly one component of
$\mathcal H(\mu,\nu_A)$. Therefore
$c(A)=\lVert\pi_A\rVert$.

Since $g_A$ preserves $U$ and $V$, let $g_{A,U}$ and $g_{A,V}$ denote
its restrictions to $U$ and $V$, respectively. From
$\nu_A=g_A\nu_0g_A^{-1}$, we obtain
$\psi_A=g_{A,U}\psi_0g_{A,V}^{-1}$ (Indeed, for \(u\in U\), we have \(ug_A=ug_{A,U}\), then \((ug_{A,U})\nu_0=(ug_{A,U})\psi_0\in V\), and \(g_A^{-1}\) acts on \(V\) as \(g_{A,V}^{-1}\). Hence \(\psi_A=g_{A,U}\psi_0g_{A,V}^{-1}\).). In other words, starting from a vertex in $U$, we first rearrange the vertices of $U$ by $g_A$, then follow the $\nu_0$-matching to $V$, and finally rearrange the vertices of $V$ by $g_A^{-1}$.   Using $\pi_A=\psi_A\phi^{-1}$ and
$\pi_\varnothing=\psi_0\phi^{-1}$, we have
$\pi_A=g_{A,U}\psi_0g_{A,V}^{-1}\phi^{-1}
=g_{A,U}\psi_0(\phi^{-1}\phi)g_{A,V}^{-1}\phi^{-1}
=g_{A,U}\pi_\varnothing
\bigl(\phi g_{A,V}^{-1}\phi^{-1}\bigr)$.

Each $\delta_i$ is a product of two transpositions and is therefore an even permutation. Since $g_A$ is a product of bi-transpositions, it is also even, so $\operatorname{sgn}(g_A)=1$. Since $g_A$ preserves $U$ and $V$, we have $g_A=g_{A,U}g_{A,V}$, where $g_{A,U}$ is equal to $g_A$ on $U$ and fixes every element of $V$, while $g_{A,V}$ is equal to $g_A$ on $V$ and fixes every element of $U$. Hence
$\operatorname{sgn}(g_A)=\operatorname{sgn}(g_{A,U})\operatorname{sgn}(g_{A,V})$.
Therefore,
$\operatorname{sgn}(g_{A,U})\operatorname{sgn}(g_{A,V})=1$.

Moreover, $\phi g_{A,V}^{-1}\phi^{-1}$ has the same sign as
$g_{A,V}$. It follows that
$\operatorname{sgn}(\pi_A)=\operatorname{sgn}(\pi_\varnothing)$
for every $A$. Recall that, for any permutation $\pi$ of $U$,
$\operatorname{sgn}(\pi)=(-1)^{|U|-\lVert\pi\rVert}$.
Therefore
$\lVert\pi_A\rVert\equiv\lVert\pi_\varnothing\rVert\pmod 2$.
Since $c(A)=\lVert\pi_A\rVert$, we obtain
$c(A)\equiv c(\varnothing)\pmod 2$ for every $A$, contradicting the
assumption that an odd difference occurs. Hence some singleton
satisfies $c(\{i\})=c(\varnothing)$. This proves the claim. \qed

\medskip
Fix an index $i$ given by the claim, and choose
$A\subseteq\{1,\ldots,m\}$ such that
$c(A)-c(\varnothing)$ is odd. We distinguish two cases.

\medskip
\noindent\textbf{Case 1. $i\notin A$.}

Let the restricted family
$\mathcal F_i^-=\{X\subseteq\{1,\ldots,m\}:i\notin X\}$.
Since the exchange operation $i$ is fixed to be omitted, the free
exchange operations are precisely those indexed by
$\mathcal P\setminus\{i\}$, with the order inherited from
$\mathcal P$. Thus
$\mathcal F_i^-=2^{\mathcal P\setminus\{i\}}$.
The ordered set $\mathcal P\setminus\{i\}$ contains $m-1$ exchange
operations, so $\mathcal F_i^-$  contains $2^{m-1}$
members.

Since $c(\varnothing)$ is a global minimum, it is also a minimum in
this restricted family. Moreover, $A\in\mathcal F_i^-$ and
$c(A)-c(\varnothing)$ is odd. Therefore the induction hypothesis
applies to the ordered set $\mathcal P\setminus\{i\}$, which has
$m-1$ exchange operations. Hence there exists
$A_0\in\mathcal F_i^-$ such that
$c(A_0)=c(\varnothing)+1$.

\medskip
\noindent\textbf{Case 2. $i\in A$.}

Let $\mathcal F_i^+=\{X\subseteq\{1,\ldots,m\}:i\in X\}$ and
$B=\{i\}$. For every $X\in\mathcal F_i^+$, we have $B\subseteq X$
and $X\triangle B=X\setminus\{i\}$. Hence
$i\notin X\triangle B$. As $X$ ranges over
$\mathcal F_i^+$, the set $X\triangle B$ ranges over all subsets of
$\mathcal P\setminus\{i\}$. Thus $\mathcal F_i^+$ is in one-to-one
correspondence with $2^{\mathcal P\setminus\{i\}}$. The ordered set
$\mathcal P\setminus\{i\}$ contains $m-1$ exchange operations, so
$\mathcal F_i^+$ contains $2^{m-1}$ members.

If $X\triangle B=\{j_1,\ldots,j_r\}$ with
$j_1<\cdots<j_r$, then Lemma~\ref{lem2:8} gives
$g_Xg_B^{-1}
=\widehat\delta_{j_1}\cdots\widehat\delta_{j_r}$.
Since $i\notin X\triangle B$, only the exchange operations indexed by
$\mathcal P\setminus\{i\}$ occur in this product. By
Lemma~\ref{lem2:8}, each $\widehat\delta_j$ is again a
bi-transposition. Hence, taking $B$ as the initial member, the free
exchange operations are indexed by the ordered set
$\mathcal P\setminus\{i\}$.

Since $i\in A$, we have $A\in\mathcal F_i^+$ and
$A\triangle B\subseteq\mathcal P\setminus\{i\}$. Moreover, by the
preceding argument, $c(B)=c(\varnothing)$, and hence
$c(A)-c(B)=c(A)-c(\varnothing)$ is odd. Since
$c(\varnothing)$ is a global minimum, $c(B)$ is also a minimum among
the members of $\mathcal F_i^+$.

Therefore the induction hypothesis applies to the ordered set
$\mathcal P\setminus\{i\}$, which has $m-1$ exchange operations. Hence
there exists a subset $D_0\subseteq\mathcal P\setminus\{i\}$ for which
the corresponding member has component number $c(B)+1$. Setting
$A_0=D_0\cup\{i\}$, we have $A_0\in\mathcal F_i^+$ and
$c(A_0)=c(B)+1=c(\varnothing)+1$.

Thus, in both cases, the required set $A_0$ exists.  The proof is complete.
\end{proof}

\begin{lemma}
\label{lem3:2}
For $A\subseteq \mathcal P$, we assume $f(Q_0)=\min_A f(Q_A)$. If $f(Q_A)-f(Q_0)$ is odd for some $A$, then there exists $A_0\subseteq\mathcal P$ such that $f(Q_{A_0})=f(Q_0)+1$.
\end{lemma}

\begin{proof}
By Lemma~\ref{lem2:13}, for each $A\subseteq\mathcal P$, we have
$\nu_A=g_A\nu_0g_A^{-1}$. By Lemma~\ref{lem2:14},
$f(Q_A)=c(A)$. Since $Q_0$ has the minimum number of faces, the
corresponding component number $c(\varnothing)$ is also a minimum.

By assumption, $f(Q_A)-f(Q_0)=c(A)-c(\varnothing)$ is odd for some
$A\subseteq\mathcal P$. Therefore, Lemma~\ref{lem3:1} gives some
$A_0\subseteq\mathcal P$ such that
$c(A_0)=c(\varnothing)+1$, equivalently,
$f(Q_{A_0})=f(Q_0)+1$. The lemma follows.
\end{proof}

\begin{proof}[\bf{Proof of Theorem \ref{t1:3}}]
 Let $Q_{\min}\in\mathcal B(\mathbf{\Pi})$ be an embedding achieving the minimum face number, and suppose that $f(Q_{\min})\equiv i\pmod{2}$. Since the opposite parity class is nonempty, there exists an embedding $Q'\in\mathcal B(\mathbf{\Pi})$ satisfying $f(Q')\equiv 1-i\pmod{2}$. By Lemma~\ref{lem2:7}, there is an ordered adjacent-exchange family $\{Q_A:A\subseteq\mathcal P\}$ with $Q_\varnothing=Q_{\min}$ and $Q_{\mathcal P}=Q'$.

In this family, $Q_\varnothing$ has the minimum face number, while $Q_{\mathcal P}$ has the opposite parity. Lemma~\ref{lem3:2} guarantees some $A\subseteq\mathcal P$ for which $f(Q_A)=f(Q_\varnothing)+1=m_i+1$, and this value has parity $1-i$. Since $m_i$ is the global minimum face number, every face number of parity $1-i$ is at least $m_i+1$. Thus, $m_i+1$ is the minimum face number in the opposite parity class, so $m_{1-i}=m_i+1$. It follows that $|m_0-m_1|=1$. From Remark \ref{re2:16}, the Theorem \ref{t1:3} follows.
\end{proof}

\section{The Proof of Theorem \ref{t1:4}}\label{Sec:4}

 As seen in Case~1 of the proof of Lemma~\ref{lem3:1}, if the exchange operation $p_i$ is fixed to be omitted, then all choices satisfy $i\notin A$, and the free choices come only from $\mathcal P\setminus\{i\}$. Similarly, in Case~2, if $p_i$ is fixed to be performed, then all choices satisfy $i\in A$.

More generally, let $\mathcal P_0,\mathcal P_1\subseteq\mathcal P$ be disjoint subsets, where the exchange operations in $\mathcal P_0$ are required to be omitted and those in $\mathcal P_1$ are required to be performed. A subset $A\subseteq\mathcal P$ is called \emph{valid} if $A\cap\mathcal P_0=\varnothing$ and $\mathcal P_1\subseteq A$. Thus the exchange operations in $\mathcal P\setminus(\mathcal P_0\cup\mathcal P_1)$ may be either performed or omitted.

\begin{proposition}\label{prop4:1}
 For a fixed valid set $A'$, a subset $A\subseteq\mathcal P$ is valid
if and only if $
A\triangle A'
\subseteq
\mathcal P\setminus(\mathcal P_0\cup\mathcal P_1).
$
\end{proposition}
\begin{proof}
Suppose first that $A$ is valid. Since both $A$ and $A'$ are valid,
we have $\mathcal P_1\subseteq A\cap A'$ and
$A\cap\mathcal P_0=A'\cap\mathcal P_0=\varnothing$. Hence $A$ and
$A'$   have the same intersection with
$\mathcal P_0\cup\mathcal P_1$, and therefore
$A\triangle A'\subseteq
\mathcal P\setminus(\mathcal P_0\cup\mathcal P_1)$.

Conversely, suppose that
$A\triangle A'\subseteq
\mathcal P\setminus(\mathcal P_0\cup\mathcal P_1)$.
Then $A$ and $A'$  have the same intersection with
$\mathcal P_0\cup\mathcal P_1$. Since $A'$ is valid, it follows that
$\mathcal P_1\subseteq A$ and
$A\cap\mathcal P_0=\varnothing$. Hence $A$ is valid. The result follows.
\end{proof}

We consider the \textit{restricted family}
$\{Q_A:  A \text{ is valid}\}$.
For each $i\in\{0,1\}$, we define the set
\[
\mathcal F_i^{\mathrm{res}} = \big\{\,f(Q_A) : A \text{ is valid and } f(Q_A)\equiv i\pmod{2}\,\big\}.
\]
If $\mathcal F_i^{\mathrm{res}}$ is nonempty, we set $m_i^{\mathrm{res}} = \min\mathcal F_i^{\mathrm{res}}$ and $M_i^{\mathrm{res}} = \max\mathcal F_i^{\mathrm{res}}$.

\begin{lemma}
\label{lem4:1}
If both $\mathcal F_0^{\mathrm{res}}$ and
$\mathcal F_1^{\mathrm{res}}$ are nonempty, then
$\bigl|m_0^{\mathrm{res}}-m_1^{\mathrm{res}}\bigr|=1$.
\end{lemma}
 \begin{proof}
Choose a valid set $A'$ such that $f(Q_{A'})$ attains the minimum
face number in the restricted family, and suppose that
$f(Q_{A'})\equiv i\pmod 2$. Thus
$m_i^{\mathrm{res}}=f(Q_{A'})$.

By Proposition~\ref{prop4:1}, as $A$ ranges over all valid sets,
$S=A\triangle A'$ ranges over all subsets of
$\mathcal P\setminus(\mathcal P_0\cup\mathcal P_1)$.
By Lemma~\ref{lem2:8}, the product
$g_Ag_{A'}^{-1}$ is an ordered product of the corresponding
bi-transpositions $\widehat\delta_j$, and each
$\widehat\delta_j$ is again a bi-transposition. Moreover,
$\nu_A=(g_Ag_{A'}^{-1})\nu_{A'}(g_Ag_{A'}^{-1})^{-1}$.
Therefore, taking $A'$ as the initial member,
the restricted family is an ordered bi-transposition family
whose free choices are precisely the
exchange operations in
$\mathcal P\setminus(\mathcal P_0\cup\mathcal P_1)$.

Since $f(Q_{A'})$ is minimal in the restricted family,
Lemma~\ref{lem2:14} implies that
$c(\mu,\nu_{A'})$ is also minimal. Since both
$\mathcal F_0^{\mathrm{res}}$ and
$\mathcal F_1^{\mathrm{res}}$ are nonempty, there exists a valid set
$A$ such that $f(Q_A)$ has parity opposite to that of $f(Q_{A'})$.
Hence $f(Q_A)-f(Q_{A'})$ is odd. By Lemma~\ref{lem2:14},
$c(\mu,\nu_A)-c(\mu,\nu_{A'})$ is also odd.
Lemma~\ref{lem3:1} therefore gives a subset of the remaining exchange
operations whose corresponding component number is
$c(\mu,\nu_{A'})+1$. Let $A_0$ be the corresponding valid set.
Then, by Lemma~\ref{lem2:14},
$f(Q_{A_0})=f(Q_{A'})+1=m_i^{\mathrm{res}}+1$.
This value has parity $1-i$.
Since
$f(Q_{A'})=m_i^{\mathrm{res}}$
is the minimum face number in the entire restricted family, every face
number of parity $1-i$ is at least
$m_i^{\mathrm{res}}+1$.
 Hence
$m_{1-i}^{\mathrm{res}}=m_i^{\mathrm{res}}+1$.
Therefore,
$\bigl|m_0^{\mathrm{res}}-m_1^{\mathrm{res}}\bigr|=1$. The result follows.
\end{proof}

For integers $a\le b$ of the same parity, let $\intervaltwo{a}{b}=\{a,a+2,\ldots,b\}.$
The following lemma asserts that there are no gaps within a single parity class $\mathcal F_i^{\mathrm{res}}$.

\begin{lemma}
\label{lem4:2}
For each $i\in\{0,1\}$, if
$\mathcal F_i^{\mathrm{res}}$
is nonempty, then $\mathcal F_i^{\mathrm{res}}=\intervaltwo{m_i^{\mathrm{res}}}{M_i^{\mathrm{res}}}$.
\end{lemma}

\begin{proof}
We prove the statement by induction on
$\left|\mathcal P\setminus(\mathcal P_0\cup\mathcal P_1)\right|$,
i.e., the number of unfixed choices.

If $\mathcal P\setminus(\mathcal P_0\cup\mathcal P_1)=\varnothing$, then all the elements of $\mathcal P$ are fixed, so there is exactly one valid set, namely $A=\mathcal P_1$. Hence the restricted family contains only one bi-rotation system $Q_{\mathcal P_1}$. For each $i\in\{0,1\}$, the set $\mathcal F_i^{\mathrm{res}}$ is therefore either empty or  a set of one-element ${f(Q_{\mathcal P_1})}$. Since  singleton set clearly has no gaps. Thus the result holds.

Suppose now that $\mathcal P\setminus(\mathcal P_0\cup\mathcal P_1)\neq\varnothing$, and assume that the result holds whenever the number of unfixed exchange operations is smaller. Choose $p\in\mathcal P\setminus(\mathcal P_0\cup\mathcal P_1)$.

Since $p$ is unfixed, a valid set $A$ may either contain $p$ or omit $p$. We therefore divide the valid sets into two families. In the first family, $p\notin A$, so $p$ is fixed to be omitted; equivalently, $\mathcal P_0$ is replaced by $\mathcal P_0\cup\{p\}$. In the second family, $p\in A$, so $p$ is fixed to be performed; equivalently, $\mathcal P_1$ is replaced by $\mathcal P_1\cup\{p\}$. In either case, the number of unfixed exchange operations decreases by one.

For $i\in\{0,1\}$, let $\mathcal F_i^-$ be the set of face numbers of parity $i$ arising from valid sets with $p\notin A$, and let $\mathcal F_i^+$ be the corresponding set arising from valid sets with $p\in A$. Explicitly,
\[
\mathcal F_i^-= \bigl\{f(Q_A):A\text{ is valid},\ p\notin A,\ f(Q_A)\equiv i\pmod 2\bigr\}
\]
and
\[
\mathcal F_i^+= \bigl\{f(Q_A):A\text{ is valid},\ p\in A,\ f(Q_A)\equiv i\pmod 2\bigr\}.
\]

Each of these two families has one fewer unfixed exchange operation. Hence, by the induction hypothesis, whenever nonempty,  $$\mathcal F_i^-= \intervaltwo{\min\mathcal F_i^-}{\max\mathcal F_i^-}~~~ \mbox{and}~~~\mathcal F_i^+= \intervaltwo{\min\mathcal F_i^+}{\max\mathcal F_i^+}.$$

If one of $\mathcal F_i^-$ and $\mathcal F_i^+$ is empty, then $\mathcal F_i^{\mathrm{res}}=\mathcal F_i^-\cup\mathcal F_i^+$ is equal to the other one, and the result follows from the induction hypothesis. Therefore, we suppose that both $\mathcal F_i^-$ and $\mathcal F_i^+$ are nonempty. Let $a_i=\min\mathcal F_i^-$ and $b_i=\min\mathcal F_i^+$.

Choose a valid set $A'$ with $p\notin A'$ and
$f(Q_{A'})=a_i$, and let $A''=A'\cup{p}$. Since $p$ is not
fixed, $A''$ is also valid by Proposition \ref{prop4:1}. Moreover,
$A'\triangle A''={p}$. By
Corollary~\ref{cor2:15},
$|f(Q_{A''})-a_i|\leq2$.

If $f(Q_{A''})\equiv i\pmod{2}$, then, by the definition of $b_i$,
$b_i\leq f(Q_{A''})\leq a_i+2$.

Suppose instead that
$f(Q_{A''})\equiv1-i\pmod{2}$. Since $f(Q_{A''})$ and $a_i$ have
opposite parity and differ by at most $2$, we have
$|f(Q_{A''})-a_i|=1$. Hence
$f(Q_{A''})\leq a_i+1$.

Let $b_{1-i}$ be the minimum face number of parity $1-i$ among the
valid sets with $p\in A$. This value exists because $Q_{A''}$ has
the parity $1-i$. Since the second family also contains a face number of
parity $i$, both parity classes are nonempty in this restricted
family. Therefore, by Lemma~\ref{lem4:1},
$|b_i-b_{1-i}|=1$. Hence
$b_i\leq b_{1-i}+1\leq f(Q_{A''})+1\leq a_i+2$.

Interchanging the roles of the two families gives
$a_i\leq b_i+2$. Thus
$|a_i-b_i|\leq2$. Since $a_i$ and $b_i$ have the same parity,
$|a_i-b_i|\in\{0,2\}$.

 Assume, without loss of generality, that $a_i\leq b_i$. Then either
$b_i=a_i$ or $b_i=a_i+2$. Hence the two step-two intervals
$\mathcal F_i^-$ and $\mathcal F_i^+$ either have the same minimum or
their minimum values are consecutive integers of parity $i$. Therefore,
their union
$\mathcal F_i^-\cup\mathcal F_i^+=\mathcal F_i^{\mathrm{res}}$
contains every integer of parity $i$ between its minimum and maximum
values.

The proof is complete.
\end{proof}

\begin{proof}[\bf{Proof of Theorem~\ref{t1:4}}]
Choose $Q_m,Q_M\in\mathcal B(\mathbf{\Pi})$ such that
$f(Q_m)=m_i$ and $f(Q_M)=M_i$. By Lemma~\ref{lem2:7}, there exists
an ordered adjacent-exchange family
$\{Q_A:A\subseteq\mathcal P\}$ with
$Q_\varnothing=Q_m$ and $Q_{\mathcal P}=Q_M$.

Applying Lemma~\ref{lem4:2} with
$\mathcal P_0=\mathcal P_1=\varnothing$, every integer of parity $i$
between $m_i$ and $M_i$ is attained by some $Q_A$. Hence
$\intervaltwo{m_i}{M_i}\subseteq\mathcal F_i$.  Conversely, every element of \(\mathcal F_i\) has parity \(i\) and lies between \(m_i\) and \(M_i\). Hence \(\mathcal F_i\subseteq\intervaltwo{m_i}{M_i}\).Therefore,
$\mathcal F_i=\intervaltwo{m_i}{M_i}$.

Theorem~\ref{t1:4} now follows from Remark~\ref{re2:16}.
\end{proof}

\section{Conclusions}\label{Sec:5}

We conclude with two conjectures concerning the shape of the Euler-genus distribution of a signed graph. Let $\Sigma$ be a signed graph, and let $g_k(\Sigma)$ denote the number of embeddings of $\Sigma$ having Euler-genus $k$. The sequence
$g_0(\Sigma),g_1(\Sigma),\ldots,g_{\gamma_M^E(\Sigma)}(\Sigma)$
is called the \emph{Euler-genus distribution} of $\Sigma$.

Recall that a finite sequence of nonnegative numbers
$b_0,b_1,\ldots,b_r$ is \emph{unimodal} if there exists an index $j$
such that
$$b_0\leq b_1\leq\cdots\leq b_j\geq b_{j+1}\geq\cdots\geq b_r,$$ and is {\it log-concave} if
$$b_k^2\geq b_{k-1}b_{k+1}~~\mbox{for}~~1\leq k\leq r-1.$$

For the  shape of genus distributions  of graphs, Gross, Robbins, and Tucker
\cite{GrossRobbinsTucker1989} conjectured that the genus distribution
of every graph is log-concave. This long-standing conjecture was
recently disproved by Mohar \cite{Mohar2026} with counterexamples. However, these counterexamples are still unimodal. The conjecture on the unimodality
of the genus distribution of graphs is still open \cite{Mohar2026}.

In the signed setting, by computing the Euler-genus distributions of some signed graphs, we found that there exist some signed graphs for which, even after removing the zero values from the Euler-genus distribution, the remaining sequence is still not unimodal. Meanwhile, from the results of this paper, we know that both the even-indexed subsequence $
\left(g_{2j}(\Sigma)\right)_
{j\geq 0}
$ and the odd-indexed subsequence $
\left(g_{2j+1}(\Sigma)\right)_
{j\geq 0}
$ are gap-free, that is, there are no zeros in the interior of the sequences. Therefore, we propose the following two conjectures. We also note that two conjectures are equivalent when both   even-indexed subsequence and odd-indexed subsequence are nonempty, that is, if one holds then so does the other.

\begin{conjecture}
For every signed graph $\Sigma$, the even-indexed subsequence
$
\left(g_{2j}(\Sigma)\right)_
{j\geq 0}
$
of its Euler-genus distribution is unimodal.
\end{conjecture}

\begin{conjecture}
For every signed graph $\Sigma$, the odd-indexed subsequence
$
\left(g_{2j+1}(\Sigma)\right)_
{j\geq 0}
$
of its Euler-genus distribution is unimodal.
\end{conjecture}

\section*{Acknowledgements}

Yichao Chen is supported by National Natural Science Foundation of China  (No. 12271392). Yan Yang is supported by
National Natural Science Foundation of China (No. 12371350).

\end{document}